\documentclass[12pt, letterpaper]{article}
\usepackage[OT4]{fontenc}
\usepackage{graphics,graphicx}
\usepackage{amsmath}
\usepackage{amssymb,amsfonts}
\usepackage {amsthm}
\usepackage{xcolor}

\usepackage{enumerate}
\usepackage{hyperref}
\usepackage[T2A]{fontenc}
\usepackage[cp1251]{inputenc}
\usepackage[all]{xy}
\usepackage[linesnumbered,boxed]{algorithm2e}

\newtheorem{theorem}{Theorem}
\newtheorem{lemma}{Lemma}
\newtheorem{proposition}{Proposition}

\theoremstyle{definition}

\newtheorem{example}{Example}
\newtheorem{remark}{Remark}

\newcommand\myfootnote[1]{
	\renewcommand{\thefootnote}{}
	\footnotetext{#1}
	\def\thefootnote{\@arabic\c@footnote}
}

\renewcommand{\subsection}{\@startsection{subsection}{2}{0mm}{-\baselineskip}{-5pt}{\it \bf}}
\makeatother

\title{Locally standard measure algebras}
\author{Oksana Bezushchak\footnote{The first author was supported by PAUSE program (France);  and partly supported by UMR 5208 du CNRS; and partly supported by MES of Ukraine: Grant for the perspective development of the scientific direction "Mathematical sciences and natural sciences" at TSNUK.}, Bogdana Oliynyk\footnote{The second author was partially supported by the grant for scientific researchers of the ``Povir u sebe'' Ukrainian Foundation.}}

\begin{document}
	
	\maketitle
	
	\small \noindent Faculty of Mechanics and Mathematics, Taras Shevchenko National University of Kyiv, Volodymyrska 60, Kyiv 01033, Ukraine\\
Faculty of Applied Mathematics, Silesian University of Technology, Kaszubska 23, Gliwice 44-100, Poland, Department of Mathematics, National University of Kyiv-Mohyla Academy, Skovorody 2, Kyiv 04070,  Ukraine

bezushchak@knu.ua, boliynyk@polsl.pl

{\it {\bf Keywords:}  measure algebra; locally matrix algebra; Boolean algebra; Hamming space; Steinitz number}

{\bf 2020}{ {\bf Mathematics Subject Classification:} 06E25, 15A30, 16S50, 28A6}


\begin{abstract}
We parameterize countable locally standard measure algebras by pairs of a Steinitz number and a real number greater or equal to $1.$ This is an analog of the theorems of J.Dixmier and A.A.Baranov.		
	\end{abstract}

\section*{Introduction}

Let $\mathbf{F}_2$ be the field of order $2.$ By a \emph{Boolean algebra} we mean an associative commutative algebra over the field $\mathbf{F}_2$ satisfying the identity $x^2 =x.$

Let $[0,\infty)$ denote the set of nonnegative real numbers. Let $H$ be a Boolean algebra. We call a function $\mu:H \rightarrow [0,\infty)$ a \emph{measure} if
\begin{enumerate}
  \item[(1)] $\mu(a)=0$ if and only if $a=0,$ $a\in H;$
  \item[(2)] if $a, b \in H$ and $ab=0,$ then $\mu(a+b)=\mu(a)+\mu(b).$
\end{enumerate}

Following A.Horn and A.Tarski \cite{Horn_Tars}, we call a Boolean algebra $H$ with a measure  $\mu:H \rightarrow [0,\infty)$ a \emph{measure algebra}. For more information on measure algebras, see \cite{Horn_Tars,Jech,Maharam}.

If $(H,\mu)$ is a measure algebra, then the distance $d_H(a,b)=\mu(a-b)$ makes it a metric space.

\begin{example}\label{Ex1} The Boolean algebra $\mathbf{St}_n =\mathbf{F}_2^n,$ $\mathbf{F}_2=\{0,1\},$ with the function $$\mu_n(x_1,\ldots,x_n)=\frac{1}{n}\big(x_1+\cdots+x_n\big) \quad \text{for all} \quad x_1,\ldots,x_n\in \{0,1\}$$ is a measure algebra. We call the measure algebra $(\mathbf{St}_n,\mu_n)$ \emph{standard}. For all elements $a,b \in \mathbf{St}_n$ the distance $d_{\mathbf{St}_n}(a,b)$ equals the number of coordinates where $a$ and $b$ differ, divided by $n.$
\end{example}

\begin{example}\label{Ex2} Let $\mathbb{N}$ be the set of positive integers. For a sequence $\mathbf{a}=(a_1,a_2, \ldots)\in \{0,1\}^{\mathbb{N}}$, define the pseudomeasure function $$\widetilde{\mu}(\mathbf{a})=\lim_{n\rightarrow \infty} \sup \frac{1}{n}\big(a_1+\cdots+a_n\big).$$ Then $I=\{\mathbf{a}\in \{0,1\}^{\mathbb{N}} \, \big| \, \widetilde{\mu}(\mathbf{a})=0\}$ is an ideal of the Boolean algebra $\mathbf{F}_2^{\mathbb{N}}.$

Consider the Boolean algebra $B=\mathbf{F}_2^{\mathbb{N}} / I$ and the measure $$\mu(\mathbf{a}+I)=\widetilde{\mu}(\mathbf{a}), \quad \mathbf{a}\in \mathbf{F}_2^{\mathbb{N}}.$$ The measure algebra $(B,\mu)$ is called \emph{Besicovich measure algebra} (see \cite{Vershik1}).
\end{example}

\begin{example}\label{Ex3}  Let $X$ be an infinite set and let $H$ be the Boolean algebra of finite subsets of $X$, including the empty one. The measure $\mu(a)=\# a,$ $a \in H$, makes $(H,\mu)$ a measure algebra. If the set $X$ is countable, then we denote the measure algebra $(H,\mu)$ as $H(\infty)$.
\end{example}

In order to introduce the next series of examples we need to start with the concept of a Steinitz number.

A   \emph{Steinitz number} (\cite{ST})   is an infinite formal
product of the form
$$ \prod_{p\in \mathbb{P}} p^{r_p}, $$
where $ \mathbb{P} $ is the set of all primes, $ r_p \in  \mathbb{N} \cup \{0,\infty\} $ for all $p\in \mathbb{P}$.
We can define the product of two Steinitz numbers by the rule:
$$ \prod_{p\in \mathbb{P}} p^{r_p} \cdot  \prod_{p\in \mathbb{P}} p^{k_p}= \prod_{p\in \mathbb{P}} p^{r_p+k_p}, \  r_p, k_p \in  \mathbb{N} \cup \{0,\infty\},  $$
where we assume, that
$$r_p+k_p=\begin{cases}
r_p+k_p, & \text{if  $r_p < \infty$ and $k_p < \infty$, } \\
\infty, & \text{in other cases}
\end{cases}.$$

By symbol $ \mathbb{SN} $ we denote the set of all Steinitz numbers. Obviously, the set of all positive integers $ \mathbb{N} $ is the subset of $ \mathbb{SN} $. The elements of the set  $\mathbb{SN} \setminus \mathbb{N} $ are called \emph{infinite Steinitz numbers}.

\begin{example}\label{Ex4}	
An infinite sequence $\mathbf{a}=(a_1,a_2,\ldots) \in \{0,1\}^{\mathbb{N}}$
is said to be {\it periodic} if there exists a positive integer $k\in \mathbb{N}$ such that the equality $a_{i}=a_{i+k}$ holds for all $i \in
\mathbb{N}$. In this case the number $k$ is called a {\it period} of the sequence
$\mathbf{a}$.

Let $u$ be a Steinitz number. A periodic sequence $\mathbf{a}$  is called {\it $u$-periodic}  if  its minimal
period is a divisor of $u$.

Let $\mathcal{H}(u)$ be the set of all $u$-periodic sequences. Clearly, $\mathcal{H}(u)$ is a Boolean subalgebra of $\{0,1\}^{\mathbb{N}}$.  The  function $$\mu_{\mathcal{H}(u)}(a_1,a_2,\ldots)=\frac{1}{k}\big(a_1 +\ldots+a_k\big),$$
where $k$ is a period of the sequence $(a_1,a_2,\ldots), $ makes $(\mathcal{H}(u),\mu_{\mathcal{H}(u)})$ a  measure algebra.
\end{example}

A  measure algebra $(H,\mu)$  is called \emph{unital} if the Boolean algebra $H$ contains $1$ and $\mu(1)=1.$ In this case it is easy to see that $\mu(H)\subseteq [0,1]$ and $1$ is the only element of measure $1.$

The measure algebras of examples \ref{Ex1}, \ref{Ex2}, \ref{Ex4} are unital. The measure algebras of  example \ref{Ex3} are not unital.

If $(H,\mu)$ is a measure algebra and $h\in H$ is a nonzero element, then $hH$ is a unital Boolean algebra. The function $$ \mu_h : hH\rightarrow [0,1], \quad \mu_h (a)=\frac{\mu(a)}{\mu(h)}, \quad a \in hH,$$ makes $H_h=(hH, \mu_h)$ a unital measure algebra.

We say that two measure algebra $(H_1,\mu_1)$ and $(H_2,\mu_2)$ are \emph{scalar equivalent} if there exists a positive number $\alpha>0$ and an isomorphism $\varphi:H_1\rightarrow H_2$ of Boolean algebras such that $\mu_2(\varphi(a))=\alpha\mu_1(a)$ for an arbitrary element $a\in H_1.$ If measure algebras are scalar equivalent and unital, then they are isomorphic.

We call a measure algebra $(H,\mu)$ \emph{locally standard} if every finite subset of $H$ is contained in a  measure subalgebra of $(H,\mu)$ that is scalar equivalent to $\mathbf{St}_n$ for some $n\geq 1$.

If the measure algebra $(H,\mu)$ is unital and locally standard, then every finite subset of $H$ is contained in a measure subalgebra that is isomorphic to $\mathbf{St}_n$ for some $n\geq 1$.

The measure algebras of Examples  \ref{Ex1}, \ref{Ex3}, \ref{Ex4} are locally standard. The Besicovich measure algebra is not locally standard because it contains elements of irrational measure.

In \cite{BezOl_4Hamming}, we showed that an arbitrary countable unital locally standard measure algebra is isomorphic to a measure algebra $H(s),$ where $s$ is a  Steinitz number (in \cite{Sushch2,BezOl_4Hamming,Ol,OlSusch},  measure algebras are called Hamming systems).

There is a parallelism between locally standard measure algebras and locally matrix algebras (\cite{Baranov2,BezOl,BezOl_4Hamming,BezOl_2,14,Glimm,Kurochkin}). Moreover, locally standard measure algebras appear as algebras of idempotents of Cartan subalgebras of locally matrix algebras (see \cite{BezOl_4Hamming}). From this point of view, the classification of unital countable locally standard measure algebras in \cite{BezOl_4Hamming} is an analog of the theorem of J.G.Glimm \cite{Glimm}.

In this paper, we prove an analog of the theorems of J.Dixmier \cite{Diskme} and A.A.Baranov \cite{Baranov2}:
\begin{center} \emph{we parameterize countable locally standard measure algebras by pairs $(s,r),$ \\ where $s$ is a Steinitz number and $r$ is a real number greater or equal to} $1.$ \end{center}

\section{Locally matrix algebras and their spectra}\label{Sec1}

In this Section, we summarize the results about locally matrix algebras and their spectra of Steinitz numbers that we will use later in the paper.

Let $\mathbb{F}$ be a field. An associative  $\mathbb{F}$-algebra $A$ is called a \emph{locally matrix algebra} if an arbitrary finite collection of elements $a_1, \cdots, a_m \in A$ is contained in a subalgebra $A'\subset A$ that is isomorphic to a matrix algebra $M_n(\mathbb{F})$ for some $n\geq 1$.  If $A\ni 1,$ then we say that $A$ is a \emph{unital locally matrix algebra}.

For a unital locally matrix algebra $A$, let $D(A)$ be the set of all positive integers $n$ such that there exists a subalgebra $A',$    $1\in A'\subset A,$ $A' \cong M_n(\mathbb{F}).$  The least common multiple of the set $D(A)$ is called the \emph{Steinitz number} $\mathbf{st}(A)$ of the algebra $A$; see \cite{BezOl}.

J.G.Glimm \cite{Glimm} showed that if $\dim_{\mathbb{F}}A \leq \aleph_0$ and $$\mathbf{st}(A)= \prod_{p_i\in \mathbb{P}} p_i, \quad \text{then} \quad A\cong \bigotimes _{p_i\in \mathbb{P}} M_{p_i}(\mathbb{F}).$$ In particular, every countable-dimensional unital locally matrix algebra is uniquely determined by its Steinitz number.

For an element $a$ of a unital locally matrix algebra $A$ choose a subalgebra $A'\subset A$ such that $1, a \in A',$  $A' \cong M_n(\mathbb{F}).$ Let $r_{A'}(a)$ be the range of the matrix $a$ in $M_n(\mathbb{F}).$ As shown by V.M.Kurochkin \cite{Kurochkin}, the ratio $$r(a)=\frac{1}{n} \ r_{A'}(a) $$ does not depend on the choice of the subalgebra $A'.$  We call $r(a)$ the \emph{relative range} of the element $a.$ If $a,b\in A$ are orthogonal idempotents, then $r(a+b)=r(a)+r(b).$

A subalgebra $H$ of the matrix algebra $M_n(\mathbb{F})$ is called a \emph{Cartan subalgebra} if $H\cong \mathbb{F}\oplus\cdots\oplus\mathbb{F}$ ($n$ summands), in other words, $H$ is spanned by $n$  pairwise orthogonal idempotents. It is well known that every Cartan  subalgebra is conjugate of the diagonal subalgebra of $M_n(\mathbb{F})$.

Let $1\in A_1 \subset A_2 \subset \cdots$ be an ascending chain of matrix subalgebras such that $A=\cup_{i=1}^{\infty}A_i$. A \emph{general Cartan subalgebra} of $A$ is the union $H= \cup_{i=1}^{\infty}H_i$, in which every $H_i$ is a Cartan subalgebra of $A_i$ and   ${1 \in H_1 \subset H_2 \subset \cdots}$.

A   \emph{Cartan subalgebra} of $A$ is   a subalgebra $H\subset A$ with decompositions $$A= \bigotimes_{i=1}^{\infty}A_i \quad \text{and} \quad H= \bigotimes_{i=1}^{\infty}H_i,$$ in which all $A_i$ are finite-dimensional matrix algebras and $H_i$ are Cartan subalgebras of $A_i$.

Every Cartan subalgebra is a general Cartan subalgebra. In \cite{BezOl_4Hamming}, it is shown that in an arbitrary  countable-dimensional unital locally matrix algebra $A$:
\begin{enumerate}
	\item[(1)] any two Cartan subalgebras are conjugate via an automorphism of $A$,
	\item[(2)] there exists a general Cartan subalgebra of $A$ that is not a Cartan subalgebra.
\end{enumerate}

In particular, not all general Cartan subalgebras are conjugate.

Let $C$ be a commutative subalgebra of a locally matrix algebra $A$ and $1 \in C$. Let $E(C)$ be the set of all idempotents from $C$  (including $0$ and $1$). For $e,f \in E(C)$,   let $ef$ and $e+f-2ef$ be their Boolean product and Boolean sum, respectively. Hence, for an arbitrary general Cartan subalgebra $C$ of $A$ the Boolean algebra $E(C)$ with the relative range function $r$ is a locally standard measure algebra.

Let $H$ be a unital locally standard measure algebra. Let $$ D(H)=\{n \geq 1 \, \big| \, 1\in H' \subset H, \ H'\cong \mathbf{St}_n \} .$$ The least common multiple of the set $D(H)$ is called the \emph{Steinitz number of $H$} and denoted as  $\mathbf{st}(H).$

In \cite{BezOl_4Hamming}, we showed that if $H$ is a countable unital locally matrix  algebra, then $H\cong \mathcal{H}(\mathbf{st}(H));$ see Example \ref{Ex4} above. In particular, every countable unital locally standard measure  algebra is uniquely determined by its Steinitz number.

If $A$ is a  countable-dimensional unital locally matrix algebra with a Cartan subalgebra $C,$ then $\mathbf{st}((C,r))=\mathbf{st}(A);$ see \cite{BezOl_4Hamming}.

Now, let $A$ be a (not necessarily unital) locally matrix algebra. For an arbitrary idempotent $0\ne e\in A$ the subalgebra $eAe$ is a unital locally matrix algebra. The subset $$\text{Spec}(A) =\{\mathbf{st}(eAe) \, | \, e\in A, \ e\ne 0, \ e^2=e \},$$ where $e$ runs through all nonzero idempotents of the algebra $A,$ is called the \emph{spectrum} of the algebra $A.$ In \cite{Spectra_Bezushchak}, we showed that if $A, B$ are countable-dimensional  locally matrix algebras, then $A\cong B$ if and only if  $\text{Spec}(A)=\text{Spec}(B).$

Let us give necessary definitions.

For a Steinitz number $s$, let $\Omega(s)$ denote the set of all natural numbers $n\in\mathbb{N}$ that divide~$s;$ and for  Steinitz numbers $s_1, $ $s_2$, we say that $s_1$ \emph{finitely divides} $s_2$ if there exists $b\in \Omega(s_2)$ such that $s_1 = s_2/b$ (we denote: $s_1\big|_{fin} s_2$).

Steinitz numbers $s_1, $ $s_2$ are \emph{rationally connected} if $s_2 = q \cdot s_1,$ where
$q$ is some rational number.

We call a subset $S\subset  \mathbb{SN} $ \emph{saturated} if
\begin{enumerate}
	\item[1)]  any two Steinitz numbers from $S$ are rationally connected;
	\item[2)]  if $s_2 \in S$ and $s_1\big|_{fin} s_2,$ then $s_1\in  S;$
	\item[3)] if $s, ns \in S,$ where $n\in \mathbb{N} ,$ then $ks \in S$ for any $k,$ $ 1 \leq k \leq n.$
\end{enumerate}

Consider some examples of saturated sets.

\begin{example}\label{Ex5_1_Spectra}
For an arbitrary natural number $n$ the set $\{1, 2,\ldots, n\}$ is saturated.
\end{example}

\begin{example}\label{Ex6_2_Spectra}
Let $s$ be a Steinitz number. The set $$S(\infty, s)  = \Big\{\, \frac{a}{b} \cdot s \ \Big| \ a\in  \mathbb{N}, \ b\in \Omega(s) \, \Big\}$$ is saturated. For an arbitrary Steinitz number $s'\in S(\infty,s)$ we have  $S(\infty,s)= S(\infty,s').$ If $s\in \mathbb{N},$ then $ S(\infty,s)=\mathbb{N}.$
\end{example}

\begin{example}\label{Ex7_3_Spectra}
Let $r$ be a real number, $1\leq r < \infty.$ Let $s$ be an infinite Steinitz number. The set
$$S(r, s)  = \Big\{ \, \frac{a}{b} \cdot s \ \Big| \ a,b\in  \mathbb{N}, \ b\in \Omega(s), \ a\leq rb \, \Big\}$$ is saturated.
\end{example}

\begin{example}\label{Ex8_4_Spectra}
Let $s$ be an infinite Steinitz number and let $r = u/v$ be a rational number; $u, v \in \mathbb{N},$ $ v \in \Omega(s).$ Then the set
$$S^{+}(r, s)  = \Big\{ \, \frac{a}{b} \cdot s \ \Big| \ a,b\in  \mathbb{N}, \ b\in \Omega(s), \ a< rb  \, \Big\}$$
is saturated.
\end{example}

Which spectra above correspond to unital algebras? A locally matrix algebra $A$ is unital if and only if $\text{Spec}(A) =
\{1, 2, \cdots , n\},$ where $n\in \mathbb{N},$ or $\text{Spec}(A) = S(r, s),$ where $ s \in \mathbb{SN} \setminus  \mathbb{N} ,$ $r = u/v,$ $ u, v \in \mathbb{N},$ $ v \in \Omega(s).$

Now following \cite{Spectra_Bezushchak}, we will construct  countable-dimensional  locally matrix algebras with the above spectra. These algebras include all countable-dimensional  locally matrix algebras.

\begin{example}\label{Ex5}  Let $M_{\infty}(\mathbb{F})$ be the algebra of $\mathbb{N}\times \mathbb{N}$ matrices over $\mathbb{F}$ having finitely many  nonzero entries. Then $\text{Spec}(M_{\infty}(\mathbb{F}))=\mathbb{N}.$
\end{example}

\begin{example}\label{Ex6}  Let $s$ be an infinite Steinitz number, $$s=\prod_{p_i\in \mathbb{P}} p_i, \quad \text{and let} \quad A(s)= \bigotimes_{p_i\in \mathbb{P}} M_{p_i}(\mathbb{F})$$ be the  countable-dimensional unital locally matrix algebra with the Steinitz number $s.$ The algebra $$A(\infty,s)=M_{\infty}(\mathbb{F})\bigotimes{}_{\mathbb{F}} A(s)\cong M_{\infty}(A(s))$$ is a countable-dimensional nonunital locally matrix algebra, $\text{Spec}(A(\infty,s))=S(\infty,s).$
\end{example}

\begin{example}\label{Ex7}  Let $s$ be an infinite Steinitz number and let  $1\leq r < \infty$ be a real number. Let $S=S(r,s)$ or $S=S^{+}(r,s).$ Choose a sequence $b_1, b_2, \ldots \in \Omega(s) $ such that $b_i$ divides $b_{i+1},$ $i\ge    1,$ and $s$ is the least common multiple of $b_i,$ $i \ge   1.$

Let $m_i=[rb_i]$ if $S=S(r,s).$ Let $$m_i = \begin{cases}
	\ [rb_i] , & \text{if} \quad rb_i\not\in \mathbb{N}, \\
	rb_i -1, & \text{if} \quad rb_i\in \mathbb{N}
\end{cases} $$ 	
if $S=S^{+}(r,s).$	
 \end{example}

For each $i \ge 1$ consider the  countable-dimensional unital locally matrix algebra $A(s/b_i)$ and let $$A_i =M_{m_i}\big( A(s/b_i) \big), \quad i \ge 1.$$ In \cite{Spectra_Bezushchak}, it is shown that the algebra $A_i$ embeds on $A_{i+1}$ as a corner. Let $$A(S) =   \bigcup_{i\ge   1} A_i.$$ Then $\text{Spec}(A(S))=S.$

\section{Classification of locally standard measure algebras}\label{Sec2}

In \cite{BezOl_4Hamming}, we showed that given two measure algebras $(H_1,\mu_1)$ and $(H_2,\mu_2)$ there exists a unique measure $\mu$ on the Boolean algebra $H_1\otimes  {}_{\mathbf{F}_2} H_2$ such that $\mu(a\otimes b)=\mu_1(a) \mu_2(b)$ for arbitrary elements $a\in H_1,$ $b\in H_2.$

\begin{remark}
In \cite{BezOl_4Hamming}, we assumed unitality of the Boolean algebra $H_1,$ $H_2.$ However, this unitality has never been used in the definition of tensor product.
\end{remark}

For a measure algebra $(H,\mu)$, consider the following set of Steinitz numbers $$\text{Spec}(H)=\{ \ \mathbf{st}(H_n) \ |  \ 0 \ne  h\in H \ \}\subseteq  \mathbb{SN}. $$ If the Boolean algebra $H$ is unital and  $\mathbf{st}(H)=s,$ then  $$\text{Spec}(H)=  \Big\{ \, \frac{a}{b} \cdot s \ \Big| \  b\in \Omega(s), \ 1 \le  a \le  b  \, \Big\}.$$ For nonzero elements $h_1, h_2 \in H$ we say that $h_1 \ge h_2$ if $h_1 H \supseteq h_2 H.$

\begin{lemma}\label{lemma1} Let  $(H,\mu)$ be a locally standard measure algebra. Then
\begin{enumerate}
	\item[1)] for arbitrary nonzero elements $h_1, h_2 \in H$ there exists an element $0\ne h_3 \in H $ such that $h_1 \le h_3,$ $h_2 \le h_3;$
		\item[2)] if $h_2 \le h_1 ,$ then $\text{\emph{Spec}}(H_{h_2})\subseteq \text{\emph{Spec}}(H_{h_1}).$
\end{enumerate}
 \end{lemma}
\begin{proof} Since the  measure algebra $(H,\mu)$ is locally standard there exists a subalgebra $H'\subset H,$ such that $h_1, h_2 \in H'$ and $(H', \mu)$ is scalar equivalent to a  standard measure algebra $(\mathbf{St}_k,\mu_k),$ $k\in \mathbb{N}. $ Let $\alpha$   be a positive real number and let $\varphi:\mathbf{St}_k \rightarrow  H'$ be an isomorphism such that $\mu (\varphi(a))=\alpha\,\mu_k(a)$ for an arbitrary element $a\in \mathbf{St}_k.$ Let $e$ be the identity of the Boolean algebra $\mathbf{St}_k,$ $h_3 =\varphi (e).$ Then $h_1, h_2 \le h_3. $

Now, suppose that  $h_2 \le h_1. $ We will show that $$\text{Spec}(H_{h_2})\subseteq \text{Spec}(H_{h_1}). $$ Let $s\in \text{Spec}(H_{h_2})$ which means that there exists an element $h\in h_2 H$ such that  $$\mathbf{st}\big((H_{h_2})_h\big)=s.$$ Clearly,  $h h_2 H=h h_1 H= hH.$ The measure function on $(H_{h_2})_h$ coincides with the  measure function on $(H_{h_1})_h.$ Hence, $s\in\text{Spec}(H_{h_1}).$ This completes the proof of the lemma.
  \end{proof}

 \begin{lemma}\label{lemma2} Let $A$ be a countable-dimensional unital locally matrix algebra. Let $C$ be a Cartan subalgebra of $A.$ Let $r:A \rightarrow [0,\infty)$ be the relative range function. Then the spectrum of the measure algebra $(E(C),r)$ coincides with the spectrum of $A.$
 \end{lemma}
\begin{proof}  In \cite{BezOl_4Hamming}, it was shown that  $\mathbf{st}\big((E(C),r)\big)=\mathbf{st}(A)=s.$ Hence, $$\text{Spec}\big((E(C),r)\big)=\text{Spec}(A)=\Big\{\, \frac{a}{b} \cdot s \ \Big| \  b\in \Omega(s)  \ 1 \le a \le b \, \Big\}.$$ This completes the proof  of the lemma.
 \end{proof}

 \begin{lemma}\label{lemma3} Let $A$ be a countable-dimensional unital locally matrix algebra, and let $0\ne e\in A$ be an idempotent. Let $C$ be a Cartan subalgebra of the algebra $eAe.$ Then there exists a Cartan subalgebra $\widetilde{C}$ of the algebra $A$ containing $C.$
\end{lemma}
\begin{proof} By Koethe's Theorem \cite{Koethe}, we will assume that $$A\cong \bigotimes_{i=1}^{\infty} M_{n_i}(\mathbb{F}).$$  There exists an integer $r \ge 1$ such that $$ e\in M_{n_i}(\mathbb{F})\otimes \cdots \otimes M_{n_r}(\mathbb{F}) \cong M_{n_1 \cdots n_r}(\mathbb{F}).$$   Let $C\,'$ be a Cartan subalgebra of $$ e\, M_{n_1 \cdots n_r}(\mathbb{F}) \, e.$$ It is easy to see that $C\,'$ is embeddable in some Cartan subalgebra $\widetilde{C}\,'$ of  the algebra $ M_{n_1 \cdots n_r}(\mathbb{F}) .$ Let  $C\,''$ be a Cartan subalgebra of the algebra $$\bigotimes_{i=r+1}^{\infty} M_{n_i}(\mathbb{F}).$$ Then $C\,' \otimes C\,''$ is a Cartan subalgebra of $eAe,$ whereas $\widetilde{C}\,' \otimes C\,''$ is a Cartan subalgebra of the algebra $A.$ We have 	$$C\,' \otimes C\,''\subseteq \widetilde{C}\,' \otimes C\,''.$$ Since all Cartan subalgebras in a countable-dimensional unital locally matrix algebra are conjugate  (see \cite{BezOl_4Hamming}),  there exists an automorphism $\varphi$ of the  algebra $eAe$ that maps $C\,' \otimes C\,''$ into $C.$ By \cite[Lemma 12]{BezOl_4Hamming}, the automorphism $\varphi$ extends to an automorphism $\widetilde{\varphi}$ of the algebra $A.$ Now,  $\widetilde{\varphi}\big(\widetilde{C}\,' \otimes C\,'' \big)$ is a Cartan subalgebra of the algebra $A$ that contains $C.$ 	 This completes the proof  of the lemma.
\end{proof}

Let $A$ be a nonunital  locally matrix algebra that is the union of the  strictly ascending chain of subalgebras $$ A_1 \subsetneqq A_2\subsetneqq \cdots ,\quad A= \bigcup_{i\ge   1} A_i, \quad A_i = e_i A e_i, \quad e_i^2 = e_i, \quad i\ge 1.  $$

Choose a   Cartan subalgebra  $C_1$ in $A_1.$ By Lemma \ref{lemma3}, there exist Cartan subalgebras $C_i \subset A_i,$ $i\ge  2,$ that form the ascending chain $C_1 \subseteq C_2\subseteq \cdots, $ and let $$C= \bigcup_{i\ge   1} C_i.$$

Let $r_{A_i}(a)$ denote the relative range of an element $a\in A_i$ in the algebra $A_i.$ Then $$r_{A_i}(e_1)=\alpha_i$$ is a rational number strictly lying between $0$ and $1.$

Let $i< j.$ In \cite{Spectra_Bezushchak}, it is shown that
\begin{equation}\label{equa1}
r_{A_j}(a)=r_{A_i}(a)\cdot r_{A_j}(e_i).
\end{equation}
For an element $a\in H$ define $$ \mu (a) = \frac{r_{A_i}(a)}{\alpha_i}.$$
If $i< j,$ then the equality (\ref{equa1}) implies $$\frac{r_{A_i}(a)}{\alpha_i}=\frac{r_{A_j}(a)}{\alpha_j},$$ which shows that the function $\mu$ is well defined.

It is easy to see that $\left( E(C), \mu \right)$ is a measure algebra. Moreover, the measure algebra $\left( E(C), \mu \right)$ is locally standard. Indeed, an arbitrary finite subset of $C$ lies in $C_i$ for some $i \ge  1.$ The measure algebra $\left( E(C_i), \mu \right)$ is scalar equivalent to the unital measure algebra  $\left( E(C_i), r_{A_i} \right),$ which has been shown to be locally standard in Sec.~\ref{Sec1}.

\begin{lemma}\label{lemma4}   $\text{\emph{Spec}}(C)=\text{\emph{Spec}}(A).$ 	
\end{lemma}
\begin{proof} For an arbitrary element $0 \ne h \in E(C)$ the space $hC$ is a Cartan subalgebra of $hAh.$ Hence (see \cite{BezOl_4Hamming}),  $\mathbf{st}\left( (E(hC),\mu) \right)=\mathbf{st}(hAh).$ We proved that $$\text{Spec}\left( (E(hC),\mu) \right) \subseteq \text{Spec}(A).$$
	
Let $0 \ne e\in A$ be an idempotent. Let $e\in A_i.$ It is easy to see that the idempotent $e$ lies in some Cartan subalgebra of the algebra $A_i .$  Since  all Cartan subalgebras in $A_i$ are conjugate there exists an automorphism $\varphi \in \text{Aut} (A_i)$ that moves $e$ to $H_i,$ $\varphi(e)\in H_i .$

Now,  $$\mathbf{st}(eAe)= \mathbf{st}(eA_i e)= \mathbf{st}\left(\varphi(e) A_i \varphi(e)\right)=\mathbf{st}\left( E(C_i)_{\varphi(e)}\right) =$$ $$ \mathbf{st}\left( E(C)_{\varphi(e)}\right) \in \text{Spec}\left( (E(C),\mu) \right).$$  This completes the proof  of the lemma.
\end{proof}

\begin{lemma}\label{lemma5}  For an arbitrary countable-dimensional locally standard measure algebra $(H, \mu)$ the spectrum $\text{\emph{Spec}}\left((H, \mu)\right)$ is a saturated set of Steinitz numbers. 		
 \end{lemma}
\begin{proof}  Since the Boolean algebra $H$ is countable by Lemma \ref{lemma1}, there exists an increasing sequence $h_1 \le h_2 \le \cdots $ of elements of $H$ such that $$ \bigcup_{i\ge   1} h_i H=H.$$ Then $$\text{Spec}(H)= \bigcup_{i=   1}^{\infty} \text{Spec}(H_{h_i}).$$
 	
Let  $s_i= \mathbf{st}(H_{h_i})$ and let $A_i =A(s_i).$ By Lemma \ref{lemma2},  $$\text{Spec}(H_{h_i})=\text{Spec}(A_{i}).$$ Hence, 	$\text{Spec}(H_{h_i})$ is a saturated set. Now, it remains to notice that a union of an ascending chain of saturated sets is a saturated set. This completes  the proof  of the lemma.
\end{proof}

We showed (see \cite{Spectra_Bezushchak}) that the spectrum of a countable locally standard measure algebra is one of the following sets of Steinitz numbers:
\begin{equation} \label{equa2}
\{1,2, \ldots, n\}; \quad  S(\infty,s), s\in \mathbb{SN};   \quad S(r,s) \ \text{or} \ S^{+}(r,s), \ \text{where} \ r\in [1,\infty),  s\in \mathbb{SN}\setminus \mathbb{N}.
\end{equation}

\begin{lemma}\label{lemma6}  For an arbitrary saturated  set $S$ of Steinitz numbers there exists a countable locally standard measure algebra $(H, \mu)$ such that  the spectrum $\text{\emph{Spec}}\left((H, \mu)\right)=S.$  		
\end{lemma}
\begin{proof} In \cite{Spectra_Bezushchak}, it was shown that there exists a countable-dimensional  locally matrix algebra $A$ such that $\text{Spec}(A)=S.$

Suppose that the algebra $A$ is unital. Let $r_A$  be the relative range function of the algebra $A.$ Let $C$ be a Cartan subalgebra of the algebra $A.$ Then the locally standard measure algebra $\left(E(C), r_A\right)$ has $S$ as a spectrum.

Now, suppose  that the algebra $A$ is not unital. Then there exists an increasing sequence of idempotents $e_1 < e_2 < \cdots $ such that $$ \bigcup_{i\ge   1} e_i A e_i =A. $$ We showed above that there exists an ascending chain of  Cartan subalgebras  $C_i \subset e_i A e_i,$ $C_1 \subset C_2 \subset \cdots$ and defined a range function $\mu: H \rightarrow [0,\infty)$ on the union $$C=  \bigcup_{i\ge   1} C_i. $$ By Lemma \ref{lemma4}, the measure algebra $\left( E(C), \mu \right)$ has $S$ as a spectrum. This completes  the proof  of the lemma.
\end{proof}

Let us give a more explicit description of locally standard measure algebras having each of the sets  (\ref{equa2}) as a spectrum.

\begin{example}\label{Ex8} For the  standard measure algebra $\mathbf{St}_n$ we have $\text{Spec}(\mathbf{St}_n)=\{ 1,2, \ldots,n\}.$
\end{example}

\begin{example}\label{Ex9} Let $$ s= \prod_{p_i\in \mathbb{P}} p_i $$ be a Steinitz number; and let $$ H(s) = \bigotimes_{p_i\in \mathbb{P}} \mathbf{St}_{p_i}$$ be the corresponding tensor product of   standard measure algebras. Let $H(\infty)$ be the   measure algebra of Example \ref{Ex3}. Then $$\text{Spec}\big(H(\infty) \otimes H(s)\big) = S(\infty,s).$$
\end{example}

\begin{example}\label{Ex10} These examples are analogs of the Examples \ref{Ex7} of locally matrix algebras. Still for completeness and for the benefit of a reader we give here full details.

Let $s$ be an infinite Steinitz number; and let $1 \le r < \infty$ be a real number. Let $S=S(r,s)$ or $S=S^{+}(r,s).$ Choose a sequence $b_1, b_2, \ldots \in \Omega(s)$ such that $b_i$ divides $b_{i+1},$ $i \le 1, $ and $s$ is the least common multiple of $b_i,$  $i \ge 1.$	Let $m_i= \left [rb_{i} \right]$ if $S=S \left( r,s \right);$ and let
\begin{equation*}
m_i = \begin{cases}
	\ [rb_i] , & \text{if} \quad rb_i\not\in \mathbb{N}, \\
	rb_i -1, & \text{if} \quad rb_i\in \mathbb{N}
\end{cases}
\end{equation*}
if $S=S^{+}( r,s).$
For each $i \geq 1$ consider the unital countable measure algebra $H \left( s/b_{i} \right).$ Let
$$M_{i}=\mathbf{St}_{m_{i}} \otimes H \left( s/b_{i} \right), \quad  i \geq 1.$$ The locally standard unital measure algebras $H \left( s/b_{i} \right)$ and $\mathbf{St}_{b_{i+1}/b_{i}} \otimes H \left( s/b_{i+1} \right)$ have equal Steinitz numbers.
Hence $$H \left( s/b_{i}\right) \cong \mathbf{St}_{b_{i+1}/b_{i}} \otimes H \left( s/b_{i+1} \right).$$ This implies
$$M_{i} = \mathbf{St}_{m_{i}} \otimes H \left( s/b_{i} \right) \cong \mathbf{St}_{m_{i}} \otimes \mathbf{St}_{b_{i+1}/b_{i}} \otimes H \left( s/b_{i+1} \right) \cong \mathbf{St}_{m_{i} \cdot \frac{b_{i+1}}{b_{i}}} \otimes H \left( s/b_{i+1} \right).$$
We have $$m_{i} \cdot \frac{b_{i+1}}{b_{i}} \leq m_{i+1}.$$ Let
$$e_{i}= ( \underbrace{1,1, \ldots, 1}_{m_{i} \cdot \frac{b_{i+1}}{b_{i}}} , 0, 0, \ldots, 0 ) \in \mathbf{St}_{m_{i+1}}.$$ Then $$\mathbf{St}_{m_{i} \cdot \frac{b_{i+1}}{b_{i}}} \cong e_{i} \mathbf{St}_{m_{i+1}}e_{i}$$ and, therefore, the measure algebra $M_{i}$ is isomorphic to the corner $( e_{i} \otimes 1 ) \,  M_{i+1} \, ( e_{i} \otimes 1 )$ of the measure algebra $M_{i+1}.$ Let $$M \left( S \right) = \bigcup_{i \geq 1} M_{i}.$$
Then $\text{Spec} \left( M \left( S \right) \right) = S.$
\end{example}

We will denote $H \left( r,s \right) = M \left( S \left( r,s \right) \right),$
$H^{+} \left( r,s \right) = M \left( S^{+} \left( r,s \right) \right).$

Now, our aim is to prove the following proposition.

\begin{proposition}\label{proposition1}
	Let $\left( H_{1},\mu_{1} \right),$ $ \left( H_{2}, \mu_{2} \right)$ be countable locally standard measure algebras. Suppose that $\text{\emph{Spec}} \left( H_{1} \right) = \text{\emph{Spec}} \left( H_{2} \right).$ Then the measure algebras $H_{1},$ $H_{2}$ are scalar equivalent.
\end{proposition}

\begin{lemma}\label{lemma7}
	Let $\left( H, \mu \right)$ be a countable unital locally standard measure algebra.
	\begin{enumerate}
		\item[(1)] Let $\mathbf{St}_{n} \subset H$ be a standard subalgebra. Then an arbitrary automorphism of the measure algebra $\mathbf{St}_{n}$ extends to an automorphism of the measure algebra $H$.
		\item[(2)] Let $0 \neq h \in H.$ Then an arbitrary automorphism of the measure algebra $\left( hH, \mu \right)$ extends to an automorphism of the measure algebra $H$.
		\item[(3)] Let $0 \neq a,$ $0 \neq b$ be elements of the Boolean algebra $H.$ Then an arbitrary automorphism $\varphi : H_{a} \rightarrow H_{b}$ extends to an automorphism of the measure algebra $H$.
		\item[(4)] Let $( H_{1}, \mu_{1}),$ $\left( H_{2}, \mu_{2} \right)$ be isomorphic countable unital locally standard measure algebras. Let $a \in H_{1},$ $b \in H_{2}$ be nonzero elements. Then an arbitrary isomorphism $\left( H_{1} \right)_{a} \rightarrow \left( H_{2} \right)_{b}$ extends to an isomorphism $H_{1} \rightarrow H_{2}.$
	\end{enumerate}
\end{lemma}

\begin{proof}
	$\left( 1 \right)$ Arguing as is the proof of Theorem 1 of \cite{BezOl_4Hamming}, we choose an ascending chain of standard measure subalgebras
	$$1 \in H^{\left( 1 \right)} \subset H^{\left( 2 \right)} \subset \ldots , \quad \bigcup_{i \geq 1} H^{\left( i \right)}=H, \quad H^{\left( 1 \right)}=\mathbf{St}_{n}.$$
By Lemma 1 of \cite{BezOl_4Hamming}, there exist standard subalgebras $\overline{H}^{( i)} \subset H^{( i+1 )}, $ $ i \geq 1,$ such that $$H^{\left( i \right)} \otimes \overline{H}^{\left( i \right)}=H^{\left( i+1 \right)}.$$ Let $$\overline{H}=\bigotimes_{i \geq 1} \overline{H}^{\left( i \right)}.$$
Then $\mathbf{St}_{n} \otimes \overline{H} = H$ and  subalgebra $\mathbf{St}_{n}$ is a tensor factor in $H$. Let $\varphi$ be an automorphism of the measure algebra $H_{n}.$ Then the automorphism $$\widetilde{\varphi}: \sum_{i} a_{i} b_{i} \rightarrow \sum_{i} \varphi \left( a_{i} \right) b_{i}, \quad a_{i} \in H_{n}, \quad b_{i} \in \overline{H},$$ extends the automorphism $\varphi.$

$\left( 2 \right)$ We have $H = hH \oplus \left( 1-h \right) H.$ Let $\varphi : hH \rightarrow hH$ be an automorphism of the measure algebra $\left( hH, \mu \right).$ Then the mapping $$ha+ \left( 1-h \right) b \rightarrow \varphi \left( ha \right) + \left( 1-h \right) b; \quad a,b \in H,$$ 	is an automorphism of the measure algebra $H.$

$\left( 3 \right)$ If $\left( \mathbf{St}_{n},\mu \right)$ is a standard measure algebra and $0 \ne e \in \mathbf{St}_{n},$ then the measure algebra $( \mathbf{St}_n )_e$ is isomorphic to the standard measure algebra  $\mathbf{St}_{\mu \left( e \right) n}.$ This implies that for a unital locally standard measure algebra $\left ( H, \mu \right)$ and a nonzero element $e \in H$ we have $$\mathbf{st} \left( H_{e} \right) = \mu \left( e \right) \mathbf{st} \left( H \right).$$

Now, let $a,b$ be nonzero elements of a unital locally standard measure algebra $\left( H, \mu \right),$ and let $\varphi : H_{a} \rightarrow H_{b}$ be an isomorphism. We have $$\mathbf{st} \left( H_{a} \right) = \mu \left( a \right) \mathbf{st} \left( H \right), \quad \mathbf{st} \left( H_{b} \right) = \mu \left( b \right) \mathbf{st} \left( H \right).$$
Since $\mathbf{St} \left( H_{a} \right) = \mathbf{St} \left( H_{b} \right)$ it follows that $\mu \left( a \right) = \mu \left( b \right).$ There exists a standard measure subalgebra $\mathbf{St}_n \subset H$ that contains $1, a, b.$ Since $\mu \left( a \right) = \mu \left( b \right)$ it follows that there exists an automorphism $\psi$ of the measure algebra $H_{n}$ such that $\psi \left( a \right) = b.$ By the part $( 1)$, the automorphism $\psi$ extends to an automorphism $\widetilde{\psi}$ of the measure algebra $H.$ Let $$\chi=\widetilde{\psi}^{-1} \circ \varphi : aH \rightarrow aH.$$ By the part $( 2)$, the automorphism $\chi$ of $aH$ extends to an automorphism $\widetilde{\chi}$ of the measure algebra $H.$ The automorphism $\widetilde{\psi} \circ \widetilde{\chi}$ extends  the isomorphism $\varphi.$

$\left( 4 \right)$ Let $\varphi: aH_{1} \rightarrow bH_{2}$ and $\psi: H_{1} \rightarrow H_{2}$ be isomorphisms. Then $$\psi^{-1} \circ \varphi : aH_{1} \rightarrow \psi^{-1} \left( b \right) H_{1}$$ is an isomorphism. By the part $( 3 )$, this isomorphism extends to an automorphism $\chi$ of the measure algebra $H_{1}$. Now, the isomorphism $\psi \circ \chi : H_{1} \rightarrow H_{2}$ extends $\varphi.$ This completes the proof of the lemma.
\end{proof}

Given two Steinitz numbers $s_{1}, s_{2}$ we say that $$s_{1} \leq s_{2} \quad \text{if} \quad s_{1}=\frac{a}{b} \cdot s_{2}, \quad b \in \Omega \left( s_{2} \right), \quad 1 \leq a \leq b.$$
	
	\begin{lemma}\label{lemma8}
		Let $\left( H, \mu \right)$ be a countable locally standard (not necessarily unital) measure algebra. Let $0 \neq h \in H, \ s = \mathbf{st} ( H_{h} ).$ Suppose that $s\,' \in \text{\emph{Spec}} \left( H \right), \ s \leq s\,'.$
		Then there exists an element $h\,' \in H$ such that $hH \subseteq h^{'}H$ and $\mathbf{st} \left( H_{h\,'} \right) = s\,'$.
		\end{lemma}
		\begin{proof}
Since $s\,' \in \text{Spec} \left( H \right)$ it follows that there exists an element $e \in H$ such that $\mathbf{St} \left( H_{e} \right) = s\,'.$There exists a measure subalgebra $\left( H_{n}, \mu \right)$ such that $h, e \in H_{n}$ and $\mu \big\vert_{H_{n}} = \alpha \mu_{n}$, where $\mu_{n}$ is the measure function of the standard measure algebra and $\alpha$ is a positive real number.	We identify $H_{n}$ with $\mathbf{F}_2^n,$ where $$\mu (i_1, \ldots, i_{n}) = \frac{\alpha}{n} \Big( i_{1}+\cdots+ i_{n} \Big); \quad i_{1}, \ldots, i_{n} = 0 \text{ or } 1.$$
	Let $e_{0}$ be the identity element of the Boolean algebra $\mathbf{St}_{n}.$ Then $$hH \subseteq e_{0}H, \quad eH \subseteq e_{0}H, \quad \mathbf{st} (H_{n}) = \mu_{n} (h) \mathbf{st} (H_{e_{0}}), \quad \mathbf{st} (H_{e}) = \mu_{n} (e) \mathbf{st} \left( H_{e_{0}} \right).$$
	Since $\mathbf{st} (H_h) \leq \mathbf{st} (H_{e})$ it follows that $\mu_{n} (h) \leq \mu_{n} (e).$
	It is easy to see that there exists an automorphism $\varphi$ of the measure algebra $H_{n}$ such that $hH_{n} \subseteq \varphi (e) H_{n}$. By Lemma \ref{lemma7} $(1)$, the automorphism $\varphi$ extends to an automorphism of the measure algebra $e_{0}H$ and, in fact, to the automorphism of $H$. This implies that $\mathbf{st} (H_{h\,'}) = s\,'$, where $h\,' = \varphi (e).$ The lemma is proved.
	\end{proof}

\begin{proof}[Proof of Proposition $\ref{proposition1}$.] 		Let $(H\,', \mu\,')$ and $(H\,'', \mu\,'')$ be countable locally standard measure algebras such that $\text{Spec} (H\,') = \text{Spec}(H\,'')$. Let's list all elements in $H\,'$ and $H\,'':$ $$ H\,'= \{ 0 = h_{1}',h_{2}', \ldots \}, \quad H\,'' = \{ 0 = h_{1}'', h_{2}', \ldots \}.$$
		We will use induction on $n$ to construct sequences of elements $$a_{1}, a_{2}, \ldots \in H\,', \quad  b_{1}, b_{2}, \ldots \in H\,''$$ such that $$a_{1}H\,' \subseteq a_{2}H\,' \subseteq \cdots \subseteq a_{n}H\,' \subseteq \cdots, \quad b_{1}H\,'' \subseteq b_{2}H\,'' \subseteq \cdots \subseteq b_{n}H\,'';$$
		$$h_{1}', \ldots, h_{n}' \in a_{n}H\,'; \quad h_{1}'', \ldots, h_{n}'' \in b_{n}H\,'' \quad \text{and} \quad \mathbf{st}\big(H_{a_{n}}'\big)=\mathbf{st}\big(H_{b_{n}}''\big) \quad \text{for all} \quad n.$$
		Let $a_{1}=0, \ b_{1}=0$. Suppose that $a_{1}, \ldots, a_{n} \in H\,'$;
		$b_{1}, \ldots, b_{n} \in H\,''$ have already been constructed. Since the measure algebra $H\,'$ is locally standard there exists a subalgebra $(\mathbf{St}_{p}, \alpha \cdot \mu_{p})$ of the measure algebra $H\,'$ such that $a_{n}, h_{n+1}' \in H_{p}$.

Let $x_{n+1}$ be the identity element of the Boolean algebra  $\mathbf{St}_p.$ Clearly, $a_n, h_{n+1}' \in x_{n+1}H\,'$ and $$\mathbf{st}\big(H_{x_{n+1}}'\big) \ge \mathbf{st}\big(H_{a_{n}}'\big).$$ Since $\mathbf{st}(H_{x_{n+1}}')\in \text{Spec}(H\,')=\text{Spec}(H\,'')$ Lemma \ref{lemma8} implies that there exists an element $y_{n+1}\in H\,''$ such that $b_n \in y_{n+1} H\, ''$ and $$\mathbf{st}\big(H_{y_{n+1}}''\big)=\mathbf{st}\big(H_{x_{n+1}}'\big).$$

Since the measure algebra $H\, ''$ is locally standard there exists a subalgebra $ (\mathbf{St}_q , \beta \cdot \mu_q)$ of $H\, ''$ such that $h_{n+1}'' , y_{n+1}\in H_q.$ Let $z_{n+1}$ be the  identity element of the Boolean algebra $H_q.$ Clearly, $$ h_{n+1}'', y_{n+1} \in z_{n+1}H\, '', \quad \mathbf{st}\big(H_{z_{n+1}}''\big) \ge \mathbf{st}\big(H_{y_{n+1}}''\big). $$ Again, by Lemma \ref{lemma8}, there exists an element $t_{n+1}\in H\, '$ such that $$x_{n+1}\in t_{n+1}H\, ', \quad \mathbf{st}\big(H_{t_{n+1}}'\big) = \mathbf{st}\big(H_{z_{n+1}}''\big) .$$ Choose $a_{n+1}=t_{n+1},$ $b_{n+1}=z_{n+1}.$ The sequences $a_1, a_2, \ldots \in H\,'$ and $b_1, b_2, \ldots \in H\,''$ have been constructed.

Repeatedly applying Lemma \ref{lemma7} (4), we get an invertible linear mapping $\varphi: H\,' \rightarrow H\,''$ such that for any $i$ $$\varphi\left(a_i H\,'\right)= b_i H\,'' $$ and the restriction of $\varphi$ to $a_i H\,'$ is an isomorphism $ H_{a_i}' \rightarrow H_{b_i}''.$ It implies that for an arbitrary $a\in a_i H\,'$ we have $$ \frac{\mu\,'(a)}{\mu\,'(a_i)} = \frac{\mu\,''\big(\varphi(a_i)\big)}{\mu\,''(b_i)}, \quad  \frac{\mu\,''\big(\varphi(a)\big)}{\mu\,'(a)} =\frac{\mu\,''(b_i)}{\mu\,'(a_i)}. $$ Hence for $i< j$ $$  \frac{\mu\,''(b_i)}{\mu\,'(a_i)}=\frac{\mu\,''(b_j)}{\mu\,'(a_j)}=\alpha > 0.$$ Now, for an arbitrary element $a\in H\,'$ $$\mu\,''\big(\varphi(a)\big) = \alpha \cdot \mu \,' (a), $$ that is, the measure algebras $H\,',$ $H\,''$ are scalar equivalent. This completes the proof of the proposition.
\end{proof}		
			
Summarizing the above, we can state the following theorem.

\begin{theorem} An arbitrary countable locally standard measure algebra is scalar equivalent to  one of the following measure algebras: $$ \mathbf{St}_n; \quad H(\infty)\otimes H(s), \ s\in \mathbb{SN} ; \quad H(r,s), \quad H^{+}(r,s) , \ r\in [0,\infty ), \quad s\in  \mathbb{SN} \setminus   \mathbb{N}.$$
\end{theorem}

\end{document}